\documentclass{article}
\usepackage{graphicx,algorithm,hyperref} 
\usepackage[a4paper]{geometry}

\usepackage{amscd}
\usepackage{amsfonts,dsfont}
\usepackage{amsmath}
\usepackage{amssymb}
\usepackage{mathrsfs}
\usepackage{mathtools}
\usepackage{amsthm}
\usepackage{todonotes}

\usepackage{booktabs}
\usepackage{multirow}
\usepackage{adjustbox}
\usepackage{pdflscape}
\usepackage{array}
\usepackage{adjustbox}
\usepackage[bottom]{footmisc}

\usepackage{bm}
\usepackage{float}

\usepackage{algpseudocode}
\newtheorem{lemma}{Lemma}
\newtheorem{theorem}[lemma]{Theorem}
\newtheorem{definition}[lemma]{Definition}
\newtheorem{remark}[lemma]{Remark}

\usepackage{color}
\newcommand{\korange}{\textcolor[rgb]{1,0.5,0}} %
\usepackage[]{todonotes}

\newcommand*{\N}{\mathbb{N}} 
\newcommand*{\R}{\mathbb{R}}
\DeclareMathOperator*{\av}{average}
\newcommand*{\gk}{g_k^s}
\newcommand*{\nk}{n_k^s}
\title{Objective-Function Free Multi-Objective Optimization: \\Rate of Convergence and Performance of \\ an Adagrad-like algorithm \footnote{During the submission of this manuscript, we learned of independent work by Gon\c{c}alves, Grapiglia and Melo \cite{Grapiglia} posted on Optimization Online on January 30th, 2026, which proposes a similar Adagrad-like algorithm for multi-objective optimization. Nevertheless, the theoretical analysis of the method  and the numerical experience are different.}}
\author{Marianna~De~Santis~~~Gabriele Eichfelder~~~Margherita Porcelli}

\author{Marianna De Santis\footnote{Dipartimento di Ingegneria dell'Informazione, Universit\`a degli Studi di Firenze, Via di Santa Marta 3, 50139 Florence, Italy and member of  the SIMAI-OPTIMA Group ({\tt marianna.desantis@unifi.it}),~ORCID 0000-0002-1189-5917.}\addtocounter{footnote}{5}
\and Gabriele Eichfelder\footnote{Institute of Mathematics, Technische Universit\"{a}t Ilmenau, Po 10 05 65, D-98684 Ilmenau, Germany ({\tt gabriele.eichfelder@tu-ilmenau.de}), ORCID 0000-0002-1938-6316.} 
\and Margherita Porcelli\footnote{Dipartimento di Ingegneria Industriale, Universit\`a degli Studi di Firenze, Viale Morgagni 40/44, 50134 Firenze, Italy,  ISTI--CNR, Italy and member of the INdAM Research Group GNCS  and of the SIMAI-OPTIMA Group ({\tt margherita.porcelli@unifi.it}) ORCID 0000-0003-0183-1204.}}

\date{\today}

\begin{document}
\maketitle

\begin{abstract}

We propose an Adagrad-like algorithm for multi-objective unconstrained  optimization that relies on the computation of a common descent direction only. Unlike classical local algorithms for multi-objective optimization, our approach does not rely on the dominance property to accept new iterates, which allows for a flexible and function-free optimization framework. New points are obtained using an adaptive stepsize that does not require neither knowledge of Lipschitz constants nor the use of line search procedures. 
The rate of convergence is analyzed and is shown to be $\mathcal{O}(1 / \sqrt{ k+1})$ with respect to the norm of the common descent direction. The method is extensively validated on a broad class of unconstrained multi-objective problems and simple multi-task learning instances, and compared against a first-order line search algorithm. Additionally, we present a preliminary study of the behavior under noisy multi-objective settings, highlighting the robustness of the method.

\end{abstract}


\section{Introduction}
Recently, a class  of algorithms for unconstrained nonconvex optimization has been proposed where  the value of the (single) objective function is never computed. This class, referred to as Objective-Function-Free Optimization (OFFO) algorithms, has recently been very popular in the context of noisy problems, in particular in deep learning applications (see~\cite{duchi2011adaptive,kingma2014adam},  among many others), where they have shown remarkable insensitivity to the noise level. Also, in these applications evaluating the loss function often requires looking at millions of data points, making its repeated computation per iteration often prohibitive. This context is particularly suitable for the OFFO framework.
 OFFO contains several well-known algorithms such as Adagrad~\cite{duchi2011adaptive,mcmahan2010adaptive}, Adam~\cite{kingma2014adam}, RMSprop~\cite{tieleman2012lecture}, ADADELTA~\cite{zeiler2012adadelta} which have been proposed and widely used in the machine learning community, emerging as state-of-the-art techniques to train neural networks. All these methods only rely on current and past gradient information to adaptively determine the step size at each iteration. The OFFO algorithms  have been introduced to unify and to extend the complexity and convergence theory of many algorithms for nonconvex problems by using a trust-region framework~\cite{gratton2024complexity,GrapigliaStella}. Indeed, the work~\cite{gratton2024complexity}  re-interpreted the deterministic Adagrad as a general trust-region method in infinity norm where, unlike standard methods, no objective is evaluated to measure the progress towards a minimizer and to enforce descent.  These methods have been extended to several classes of problems including constrained optimization~\cite{bellavia2025fast,bellavia2026,gratton2024bounds}, multilevel optimization~\cite{gratton2023multilevel}, and neural network pruning applications~\cite{prunAdag}, including both deterministic and stochastic approaches.

In this work, we introduce the further challenge of optimizing multiple objective functions without evaluating any element function. This framework is particularly motivated by multi-task learning in neural networks, where a single model must simultaneously optimize multiple task-specific losses, see, e.g.,~\cite{Multi-Task,StochMO24Vicente}.  
More broadly, stochastic multi-objective optimization problems - characterized by the presence of noise and multiple competing objectives - arise in several applications, including finance, energy, transportation, logistics, and supply chain management, see, for instance, the survey~\cite{StochMO2016survey} and the references therein.  
 \footnote{Although noisy problems are a relevant motivation for our proposal, we consider here deterministic (noiseless) algorithms which are crucial to understand the behavior of stochastic ones, and provide some numerical experience on noisy multi-objective problems.} 

The majority of local algorithms for multi-objective optimization uses the fact that if a point does not satisfy suitable necessary optimality conditions, then it can be easily improved, in the sense that a new point can be easily defined, with respect to a specific quality measure, which is usually the objective functions value. One aims at a new point which improves the value of at least one objective function (sometimes also all objective functions) and which is thus not dominated by the previous one.
Most of the algorithms proposed in this respect, both for unconstrained and constrained multi-objective problems, extend the classical iterative scalar optimization algorithms, such as the steepest descent~\cite{Fliege00}, Newton~\cite{Fliege09,gonccalves2022globally}, external penalty~\cite{fukuda2016external} or interior point~\cite{fukuda2020barrier}, just to name a few. The mentioned approaches produce sequences of points able to converge to single Pareto critical points. In particular, at every iteration, these algorithms look for a new improved point, that is, a point that  dominates the previous one. Designing an algorithm that does not rely on the dominance property to accept new points is then a particular challenge in the context of multi-objective optimization.

Our framework is inspired by the OFFO work~\cite{gratton2024complexity} by leveraging on the Adagrad-Norm algorithm~\cite{Adagrad-Norm}  for deterministic non-convex optimization that relies on $\ell_2$-norm trust-region subproblems. 
Our contribution can be summarized as follows:
\begin{enumerate} 
    \item  We develop an adaptive strategy that uses a first-order common descent direction for which we provide a thorough theoretical characterization. The adaptive stepsize rule relies on the current and past common descent directions.

    \item No knowledge of the Lipschitz constants neither the employment of a line search procedure (cfr~\cite{StochMO24Vicente}) is needed to define the adaptive stepsize.

    \item We prove the convergence of the method with global convergence rate\footnote{Recall that given two sequences $\{\alpha_k\}$ and $\{\beta_k\}$ of nonnegative and positive 
     scalars, respectively,  we  say that $\alpha_k$ is $\mathcal{O}(\beta_k)$ if there exists a finite constant $\kappa\in\R$ such that $\lim_{k\rightarrow\infty} (\alpha_k/\beta_k) \leq \kappa$.} of $\mathcal{O}(1 / \sqrt{ k+1})$ with respect to the norm of the common descent direction. 
    
    \item We extensively validate our method on a large class of unconstrained multi-objective problems and on simple multi-task problems, in comparison with a first-order line search algorithm. 

    \item We provide a preliminary study on the behavior of both our method and the line search based one on noisy multi-objective problems.
\end{enumerate}\medskip

The paper is organized as follows. In Section \ref{sec:prel}, we review basic results on multi-objective optimization. Section~\ref{sec:MO-Adagrad} presents the \texttt{MO-Adagrad} algorithm and its theoretical analysis, including results on the common descent direction and a global convergence rate of $\mathcal{O}(1 / \sqrt{k+1})$. In Section \ref{sec:num}, we report our numerical experience. Finally, Section~\ref{sec:concl} concludes the paper.

\section{Basics of Multi-objective Optimization}\label{sec:prel}

We consider the unconstrained multi-objective optimization problem
\begin{equation}\label{P}\tag{MOP}
\min_{x\in\R^n}f(x)
\end{equation}
with $f\colon\R^n\to\R^m$, and each $f_j\colon\R^n\to\R$, $j\in\{1,\ldots,m\}$ being continuously differentiable. We use for any $m\in\N$ the set of indices $[m]:=\{1,\ldots,m\}$. We denote the gradient of one of the functions $f_j$ in some point $x\in\R^n$ by  $\nabla f_j(x)$.
Moreover, we assume that the multi-objective optimization problem \eqref{P} has at least one weakly efficient point. Recall that a   point $\bar x\in\R^n$  is called efficient for  \eqref{P} if for any $x\in \R^n$ the conditions  $f_j(x)\leq f_j(\bar x)$ for all $j\in[m]$ already imply $f(x)=f(\bar x)$, and it is called weakly efficient in case there is no $x\in \R^n$ with $f_j(x)< f_j(\bar x)$ for all $j\in[m]$. For vectors $y^1,y^2\in\R^m$ the inequalities $y^1\leq y^2$ and $y^1<y^2$ are understood componentwise.

In general, the task is to find efficient solutions for \eqref{P}. Weak efficiency is a weaker notion, and we have that any efficient point for \eqref{P} is also weakly efficient for \eqref{P}. Moreover, analogously to single-objective optimization, one can define local concepts, like locally weakly efficient and locally efficient points: a   point $\bar x\in\R^n$  is called locally efficient for  \eqref{P} if there exists a neighborhood $\mathcal{N}(\bar x)$ of $\bar x$ such that for any $x\in \mathcal{N}(\bar x)$ the condition  $f(x)\leq f(\bar x)$  already implies $f(x)=f(\bar x)$, and it is called locally weakly efficient in case  there exists a neighborhood $\mathcal{N}(\bar x)$ of $\bar x$ such that there is no $x\in \mathcal{N}(\bar x)$ with $f(x)< f(\bar x)$.

Additionally, we assume, as it is often done in the literature for descent-type algorithms for multi-objective optimization, that the gradients of the functions $f_j$, $j\in[m]$ are Lipschitz continuous, i.e., for each $j\in[m]$ there exists a Lipschitz constant $L_j>0$ with 
\[
\|\nabla f_j(x)-\nabla f_j(x')\|_2\leq L_j\,\|x-x'\|_2\ \mbox{ for all } x,x'\in\R^n.
\]
We will make use of 
\begin{equation}\label{eq:Lmax}
L_{\max} := \max_{j\in [m]} L_j,
\end{equation}
which denotes the maximum among the Lipschitz constants of the gradients $\nabla f_j$, $j\in [m]$.

In our proofs, the function $\Phi\colon\R^n\to\R$ will play an important role, which we define by  
\begin{equation}\label{eq:Phi}
\Phi(x) :=\displaystyle \max_{j\in [m]} f_j(x).
\end{equation}
This function is often used in the context of multi-objective trust-region methods \cite{Villacorta}. There,  for accepting a candidate point as a  new iteration point one requires a sufficient decrease for this function $\Phi$. This is a weaker criterion compared to requiring  that a sufficient decrease is reached in each objective function $f_j$. 
The function $\Phi$ is bounded from below under our assumption:

\begin{lemma}\label{lemma:philow}
Under our assumption that \eqref{P} has a weakly efficient point $\bar x \in\R^n$ there exists $\Phi_{\rm low}\in \R$ such that 
\[\inf_{x\in\R^n}\Phi(x)\geq \Phi_{\rm low}>-\infty.\]
\end{lemma}
\begin{proof}
Define $\Phi_{\rm low}:=\min_{j\in[m]}f_j(\bar x)$ and assume there exists a sequence $(x^k)_{k\in\N}\subseteq\R^n$ with $\lim_{k\to\infty}\Phi(x^k)=-\infty$, i.e., we have for all $j\in [m]$ that $\lim_{k\to\infty}f_j(x^k)=-\infty$. Then there exists $N\in\N$ such that for all $j\in [m]$ we have $f_j(x^N)<\Phi_{\rm low}$   and thus $f_j(x^N)<f_j(\bar x)$, in contradiction to $\bar x$ weakly efficient for \eqref{P}.
\end{proof}

In this work, $\Phi$ will be used for deriving the convergence analysis, but will never be used in the proposed algorithm.

A necessary condition for a point $x\in\R^n$ for being weakly efficient for \eqref{P} is that $x$  is a Pareto critical point for \eqref{P} as defined next, see, among others, \cite{Fliege00}. The necessary and sufficient optimality conditions will be formulated in Lemma \ref{lem:necandsuf}.
\begin{definition}
    We call $x\in\R^n$ a Pareto critical point for \eqref{P} if there are scalars $\lambda_j\geq 0$, $j\in[m]$ with $\sum_{j=1}^m\lambda_j=1$ and 
    \[\sum_{j=1}^m\lambda_j\nabla f_j(x)=0.\]
\end{definition}
Within numerical algorithms,  the following function $\omega$  can be used for characterizing Pareto critical points. It is also called a proximity measure in the literature on constrained multi-objective optimization, e.g. in \cite{Eichfelder2020Proximity}, then taking also the constraint functions into account.

\begin{definition}
We define the criticality  measure $\omega\colon\R^n\to \R_+$ as the function which associates to each $x\in\R^n$ the optimal value of 
\begin{equation}\label{P3}\tag{$\Omega(x)$}
    \begin{array}{rl}
    \min\limits_{\lambda\in\R^m}\ &\|\sum\limits_{j=1}^m \lambda_j\nabla f_j(x)\|_2^2\\
  \mbox{s.t.}   &  \sum\limits_{j=1}^m \lambda_j=1, \quad
    \lambda\geq 0.
    \end{array}
\end{equation}
\end{definition}

The following result is easy to see:
\begin{lemma}\label{lemma:optim}
A point $x\in\R^n$ is Pareto critical for \eqref{P} if and only if $\omega(x)=0$. 
\end{lemma}

In the literature, one can often find the following characterization for Pareto criticality of a point $x\in\R^n$, see \cite[Def. 2.2]{Fliege00}:
\begin{equation}\label{eq:critical}
\forall \ d\in\R^n\ \exists \ j\in[m]:\ \nabla f_j(x)^\top d\geq 0.
\end{equation}   
This condition is equivalent to that the optimal value of 
the optimization problem 
\begin{equation}\label{P1}\tag{mgrad(x)}
    \min\limits_{\|d\|_2^2\leq 1}\ \max\limits_{j\in[m]}\nabla  f_j(x)^\top d
\end{equation}
is zero with optimal solution $d=0$, see \cite[Lemma 3]{Fliege00}. The problem \eqref{P1} was, for instance, used in \cite{ThomannPhD19,ASMOP}. We show that this characterization is in fact equivalent to our definition above. While this is somewhat widely known, we give   a formal proof below.
\begin{lemma}\label{lemma:equiv}
A point $x\in\R^n$ is a Pareto critical point for \eqref{P} if and only if \eqref{eq:critical} holds.
\end{lemma}

\begin{proof} As we could not find a proof in the easy accessible literature of this well-known result, we give it here for completeness. A proof using Gordon's Theorem can be found in the Master's thesis \cite{ThomannMaster}.
First, assume that \eqref{eq:critical} holds, i.e., $d=0$ is an optimal solution of \eqref{P1} and thus $(0,0)$ is an optimal solution of 
\begin{equation}\label{P1'}\tag{mgrad'(x)}
    \begin{array}{rl} \min\limits_{(t,d)\in\R^{1+n}}& t \\\mbox{s.t.}&
    \nabla  f_j(x)^\top d-t\leq 0,\ \forall\ j\in[m],\\
&  \|d\|_2^2\leq 1.
    \end{array}
\end{equation}
 For \eqref{P1'}, the Slater constraint qualification  holds, and thus any optimal solution has to satisfy the KKT optimality conditions. Note that the constraint $ \|d\|_2^2\leq 1$ is not active in the optimal solution $(0,0)$ while all the other constraints are active. Thus, there exist multipliers $\lambda_j\geq 0$, $j\in[m]$, with 
\begin{equation}\label{eq:KKTcritical} 
1+\sum\limits_{j=1}^m \lambda_j(-1)=0\
\ \mbox{ and }\ 
0+\sum\limits_{j=1}^m \lambda_j \nabla  f_j(x)=0. 
\end{equation}
As a consequence,  the point $x$ is a Pareto critical point for \eqref{P}. On the other hand, if $x$ is Pareto critical for \eqref{P}, then  the conditions \eqref{eq:KKTcritical} are satisfied for some $\lambda_j\geq 0$, $j\in[m]$. Moreover, for  $(0,0)$  all constraints of the form $\nabla  f_j(x)^\top d-t\leq 0$ are active, and thus the complementarity conditions are satisfied. As \eqref{P1'} is a convex optimization problem,   this is sufficient for $(0,0)$ being an optimal solution of \eqref{P1'}. Hence, the condition \eqref{eq:critical} is satisfied.  
\end{proof}

As a consequence,  the optimal value of \eqref{P1} is non-positive, and zero if and only if $x\in\R^n$ is Pareto critical for \eqref{P}.
Moreover, by  Lemma \ref{lemma:equiv} and, for instance, \cite[Theorem 3.1]{Fliege09}, we have the following necessary and sufficient optimality condition.  

\begin{lemma}\label{lem:necandsuf}
If a point $x\in\R^n$  is locally weakly efficient for \eqref{P}, then $x$ is Pareto critical. If the functions $f_j$, $j\in[m]$ are convex and $x\in\R^n$ is Pareto critical for~\eqref{P}, then $x$ is weakly efficient for~\eqref{P}. If the functions $f_j$, $j\in[m]$ are even strictly convex and $x\in\R^n$ is Pareto critical for~\eqref{P}, then $x$ is  efficient for~\eqref{P}.
\end{lemma}

Thereby the last statement was discussed for instance in \cite{BurachikKayaRizvi2017Scalarization}, and it can easily be shown that under these conditions any weakly efficient point for \eqref{P} is also efficient for \eqref{P}.

\section{The Multi-Objective Adagrad Algorithm}\label{sec:MO-Adagrad}
Building on ideas originally developed for single-objective objective-function-free optimization, we propose the algorithm outlined in Algorithm~\ref{alg:MOADG}, named~\texttt{MO-Adagrad}. At every iteration $k$, after computing a common descent direction for the objective functions, denoted by $-g^s_k\in \R^n$ and computed by solving~\eqref{P3} in $x=x^k$ (see Step \ref{step:gsk}), the update step $s^k$ is obtained by scaling this direction with suitable weights $w_k$.
These weights (see Step \ref{step:wk}) accumulate the $\ell_2$-norm of past descent directions in an analogous fashion as in the Adagrad-norm~\cite{Adagrad-Norm} algorithm.
\begin{algorithm}{}
  \caption{\texttt{MO-Adagrad}}\label{alg:MOADG} 
  	\begin{algorithmic}[1]
    \State Initialization: a starting point $x^{0} \in \R^n$ and  a constant $\varsigma\in(0,1)$ are given. Set $k=0$ and   the initial weight scalar $w_{-1} =\sqrt{\varsigma}$.    
    
  \State Solve \eqref{P3} for $x = x^k$ with optimal solution $\lambda^k$ to obtain $g^s_k\in \R^n$ as in \eqref{eq:defdescentd}\label{step:gsk}
  
  \State Compute the weights \label{step:wk}
  \begin{equation}\label{weig}
        w_{k} = \sqrt{(w_{k-1})^2+ \|g^s_{k}\|_2^2}
  \end{equation}
  \State  Compute the step
   \[  s^k = - \frac{1}{w_k}g^s_k\] 
  \State   Define
   $$  \qquad x^{k+1} = x^k + s^k,$$ 
    increment $k$ by one and return to Step~2.
  \end{algorithmic}
  \end{algorithm}

In the following, we formally define our direction $g_k^s$ and report a number of theoretical results related to it which are needed for the subsequent convergence rate analysis.

\begin{definition}
  Let  $x^k\in\R^n$ be any iteration point and denote with $\lambda^k\in\R^m$ an  optimal solution of \eqref{P3} for $x=x^k$. Then we define the vector $\gk\in\R^n$ by 
\begin{equation}\label{eq:defdescentd}
\gk:=\sum_{j=1}^m\lambda^k_j\nabla f_j(x^k)
\end{equation}
and, in case of $\omega(x^k)>0$,  the negative of the normed vector of $g_k^s$ by 
\[\nk:=-\frac{1}{\|\gk \|_2}\gk.\]
\end{definition}
Observe that $\|\gk\|_2^2=\omega(x^k)$ and thus, by Lemma \ref{lemma:optim}, $\gk=0$ if and only if $\omega(x^k)=0$.

\begin{remark}
Note that when solving the problem~\eqref{P3} one searches for a vector with smallest Euclidean norm in the convex hull of the gradients of the functions $f_j$ in $x$. As the norm is strictly convex, this vector, which is exactly $\gk$ from the definition above, is uniquely defined. 
\end{remark}

For the following, we define a function $h\colon\R^n\to\R$ for some fixed $x\in\R^n$ by 
\begin{equation}\label{eq:defh}
    h(d):=\left(\max_{j\in[m]}\nabla  f_j(x)^\top d\right)+\frac{1}{2} \|d\|^2_2.
\end{equation}
For the convergence analysis of our algorithm it is important to relate 
\[\max_{j\in[m]}\nabla  f_j(x^k)^\top (-\gk),\] that is, the objective function of~\eqref{P1} evaluated at $x=x^k$ in $d=-\gk$, with the quantity $\omega(x^k)$.  The needed result is stated in the forthcoming Lemma~\ref{lemma:key}. For the proof of this Lemma~\ref{lemma:key}, we need to first relate the value of the function $h$ as defined in (\ref{eq:defh}) when evaluated at $x=x^k$ in $d=-\gk$ with $\omega(x^k)$. We also clarify that the obtained directions $-\gk$ and $\nk$ are in fact descent directions in case $x^k$ is not yet a Pareto critical point. 

\begin{lemma}\label{lemma:P1andP3}
Let $x^k\in\R^n$ be any iteration point which is not a Pareto critical point for \eqref{P}, that is, $\omega(x^k)>0$. Then the following statements hold.
\begin{itemize}
\item[(a)]
The direction $\nk$ is an optimal solution of \eqref{P1}
for $x=x^k$.  
\item[(b)]
The direction $-\gk$ is the unique  optimal solution of  
\begin{equation}\label{P2}\tag{pgrad(x)}
    \min\limits_{d\in\R^n}\ h(d)
\end{equation}
for $x=x^k$, and for the optimal value $h(-\gk)$ we have 
\begin{equation}\label{eq:optvrelation}
  \omega(x^k)=\|\gk\|_2^2=-2h(-\gk)\korange{,}
\end{equation}
with $h$ as defined in \eqref{eq:defh}.
    \item[(c)]
The direction $\nk$, and thus also $-\gk$, is a common descent direction of $f$ in $x^k$, i.e., there exists $t_0>0$ such that 
\[\forall j\in[m],\ \forall t'\in(0,t_0]:\ f_j(x^k+t'\,\nk)<f_j(x^k).\]
\end{itemize}
\end{lemma}
\begin{proof}

We start with proving (a). To this end, we set \[t^k:=\max\limits_{j\in[m]}\nabla  f_j(x^k)^\top \nk, \] 
and we  show that $(t^k,\nk)$   is feasible for the optimization problem \eqref{P1'} for $x=x^k$ 
and satisfies the KKT optimality conditions for  this problem  with Lagrange multipliers $\lambda_j^k$, $j\in[m]$ and $\xi^k:=\|\gk\|_2/2\geq 0$. As \eqref{P1'} is a convex optimization problem, this is sufficient for $(t^k,\nk)$  being optimal for \eqref{P1'} and thus for $\nk$ being optimal for \eqref{P1} for $x=x^k$. 

First, note that $\|\nk\|_2^2=1$. The Lagrange function reads as
\[L(t,d,\lambda^k,\xi^k)=t+\sum_{j=1}^m\lambda_j^k(\nabla f_j(x^k)^\top d-t)+\xi^k(d^\top d-1).\]
As  $\lambda^k$ is feasible for \eqref{P3} for $x=x^k$ we obtain that the partial  derivative of $L$ as defined above  w.r.t.\ $t$ equals zero.
For the partial derivatives w.r.t.\ the components of $d$ we obtain the condition 
\[ \sum_{j=1}^m\lambda_j^k \nabla f_j(x^k) +2\xi^kd=0, \]
which is satisfied for $d=\nk$ by the definition of $\xi^k$ and of $\nk$.
 It remains to show that the complementarity conditions are  fulfilled, i.e., that the Lagrange-multipliers to the inactive constraints are zero. As $\|\nk\|_2^2=1$ it remains to show that for  all $j\in[m]$ we have 
\begin{equation}\label{eq:cc}
\lambda_j^k(\nabla f_j(x^k)^\top \nk-t^k)=0.
\end{equation}
Let $J^k:=\{j\in[m]\mid  t^k=\nabla f_j(x^k)^\top \nk\}$. Note that $J^k\not=\emptyset$.
Now assume there is an index  $\ell\in [m]\setminus J^k$ with $\lambda_{\ell}^k>0$ and let $\ell'\in J^k$. Then this implies  
$ \nabla f_{\ell}(x^k)^\top \nk< \nabla f_{\ell'}(x^k)^\top \nk=t^k$.
By the definition of $\nk$ we get 
\[\gamma:=(\nabla f_{\ell'}(x^k)-\nabla f_{\ell}(x^k))^\top  \gk<0.\]
Next we define for any  $\varepsilon \in(0,\lambda_{\ell}^k)$
the vector $\tilde\lambda\geq 0$ by $\tilde\lambda_{j}:=\lambda^k_{j}$ for all $j\in[m]\setminus\{\ell,\ell'\}$  and 
\[ 
\tilde\lambda_{\ell}:=\lambda^k_{\ell}-\varepsilon,\qquad 
\tilde\lambda_{\ell'}:=\lambda^k_{\ell'}+\varepsilon.
\]
Then $\tilde\lambda$ is feasible for \eqref{P3} for $x=x^k$ as $\lambda^k$ is feasible for \eqref{P3} for $x=x^k$. We get for the objective function value of the problem \eqref{P3}  for $x=x^k$ that 
\[\begin{array}{rcl}
\|\sum\limits_{j=1}^m \tilde \lambda_j\nabla f_j(x^k)\|_2^2
&=&
\|(\sum\limits_{j=1}^m  \lambda_j^k\nabla f_j(x^k))+\varepsilon \nabla f_{\ell'}(x^k)-\varepsilon \nabla f_{\ell}(x^k)\|_2^2\\
&=&
\|\gk+\varepsilon (\nabla f_{\ell'}(x^k)-\nabla f_{\ell}(x^k))\|_2^2\\[2ex]
&=& \|\gk\|_2^2+ \varepsilon^2 \|\nabla f_{\ell'}(x^k)-\nabla f_{\ell}(x^k)\|_2^2+2\varepsilon (\nabla f_{\ell'}(x^k)-\nabla f_{\ell}(x^k))^\top \gk
\\[2ex]
&=& \|\gk\|_2^2+ \varepsilon(\varepsilon \|\nabla f_{\ell'}(x^k)-\nabla f_{\ell}(x^k)\|_2^2+2\gamma ).
\end{array}
\]
Recall  that $\gamma<0$. 
In case  $\|\nabla f_{\ell'}(x^k)-\nabla f_{\ell}(x^k)\|_2^2=0$ this is a contradiction to $ \|\gk\|_2^2$ being the optimal value of \eqref{P3} for $x=x^k$. Otherwise, for $\varepsilon$ small enough, 
we also get a contradiction to  $ \|\gk\|_2^2$ 
being the optimal value of \eqref{P3} for $x=x^k$. Thus, the assertion is proven.

For proving part (b), instead of \eqref{P2} we examine the following optimization problem 
\begin{equation}\label{P2'}\tag{pgrad'(x)}
    \begin{array}{rl} \min\limits_{(t,d)\in\R^{1+n}}& t+\frac{1}{2} \|d\|_2^2 \\\mbox{s.t.}&
     \nabla  f_j(x)^\top d-t\leq 0,\  \forall \ j\in[m] 
    \end{array}
\end{equation} 
for $x=x^k$. As this problem has a strictly convex objective function, there is not more than one optimal solution $(\bar t^k,\bar d^k)$ and we aim at showing that this is the   point $\bar t^k:=\max_{j\in[m]} f_j(x^k)^\top(-\gk)$ and $\bar d^k:=-\gk$, which is obviously feasible. We do this by showing that the KKT conditions are satisfied, as those are sufficient for this convex optimization problem. It is easy to verify that in $(\bar t^k,\bar d^k)$ with  the Lagrange-multipliers $\lambda^k\geq 0$ the derivative of the Lagrange function w.r.t.\ $t$ and $x$ vanishes. To see this, recall that $\sum_{j=1}^m\lambda^k_j=1$ as $\lambda^k$ is an optimal solution of \eqref{P3} for $x=x^k$. Next, note that the  condition  \eqref{eq:cc} can be written as follows, after multiplying the equation with $\|\gk\|_2$,
\[
\lambda_j^k\left(\nabla f_j(x^k)^\top (-\gk)-\bar t^k\right)=0\]
for all $j\in[m]$. 
With the same argumentation as in the proof of (a) we can argue that this condition has to be fulfilled, which is the complementarity condition  of \eqref{P2'}  for $x=x^k$ in the point $(\bar t^k,\bar d^k)$.  Thus, $(\bar t^k,\bar d^k)$
is a KKT point of \eqref{P2'} and hence the unique optimal solution, for $x=x^k$.  As a consequence, $\bar d^k=-\gk$ is the unique optimal solution of \eqref{P2} for $x=x^k$. 

It remains to show that \eqref{eq:optvrelation} holds, and we do this by showing that the dual problem of \eqref{P2'} is (next to some scalar multiplications and sign changes) the problem \eqref{P3}. In \cite{StochMO24Vicente} it was   already claimed that these problems are dual to each other, but the factor $-0.5$ has to be taken into account, see below.    The formulation of the dual problem of 
\eqref{P2'}, including this factor, can be found in \cite{chen22} on pages 413-414, and the calculation of $\bar d(\lambda)$ below can also be found in \cite{FliegeVicente}. We give a rigorous proof here for completeness. 

We always fix, as before, $x=x^k$.  For formulating the objective function of the  dual problem to \eqref{P2'} we need to minimize the  Langrange-function w.r.t.\ $t\in\R$ and $d\in\R^n$ for fixed $\lambda\geq 0$. As the Lagrange-function to \eqref{P2'} is defined by 
\[
\begin{array}{rcl}
L(t,d,\lambda)&=&t+\frac{1}{2}d^\top d+\sum_{j=1}^m \lambda_j(\nabla f_j(x^k)^\top d-t)
\\&=&t\,(1-\sum_{j=1}^m \lambda_j)+\frac{1}{2}d^\top d+\sum_{j=1}^m \lambda_j \nabla f_j(x^k)^\top d, 
\end{array}
\]
we get that the infimum is $-\infty$ unless $\sum_{j=1}^m \lambda_j=1$. 
In case of a $\lambda\geq 0$ with  $\sum_{j=1}^m \lambda_j=1$ the infimum  is attained in any $\bar t(\lambda)\in\R$ and 
\[\bar d(\lambda):=-\sum_{j=1}^m \lambda_j\nabla f_j(x^k).\] 
Thus the objective function of the dual problem reads, for any  $\lambda\geq 0$ with  $\sum_{j=1}^m \lambda_j=1$, as
\[q(\lambda):=\inf_{(t,d)\in\R\times \R^n}L(t,d,\lambda)= \frac{1}{2}\bar d(\lambda)^\top \bar d(\lambda)+\sum_{j=1}^m \lambda_j\nabla f_j(x^k)^\top \bar d(\lambda)=-\frac{1}{2}\bar d(\lambda)^\top\bar d(\lambda).\]
Thus, the dual problem can be formulated as 
\[
\begin{array}{rl}
    \max\limits_{\lambda\in\R^m}\ &-\frac{1}{2}\|\bar d(\lambda)\|^2_2\\[2ex]
  \mbox{s.t.}   &  \sum_{j=1}^m \lambda_j=1,\ \lambda\geq 0,
    \end{array}
\]
which corresponds to 
\[
\begin{array}{rl}
    -\frac{1}{2}\min\limits_{\lambda\in\R^m}\ &\|\sum_{j=1}^m \lambda_j\nabla f_j(x^k)\|^2_2\\[2ex]
  \mbox{s.t.}   &  \sum_{j=1}^m \lambda_j=1,\ \lambda\geq 0
    \end{array}
\]
with optimal value, cf.\ \eqref{P3} for $x=x^k$, equal to $-0.5\omega(x^k)=-0.5\|\gk\|^2_2$.
As \eqref{P2'} is convex and the Slater constraint qualification holds, we have strong duality. The optimal value of \eqref{P2'} is $h(-\gk)$, and thus the assertion is shown.

For part (c), as $\nk$ is optimal for \eqref{P1} for $x=x^k$ by (a) and as the optimal value of \eqref{P1} for $x=x^k$  has to be negative by Lemma \ref{lemma:equiv} and by the non-criticality of $x^k$, we have 
$ \max_{j\in[m]}\nabla  f_j(x^k)^\top  \nk <0$.
Note that it suffices to have that for all $j\in[m]$  it holds  $\nabla f_j(x^k)^\top \nk<0$ 
to have that $\nk$ is a common descent direction, cf. \cite[Lemma 2.5]{ThomannPhD19}. 
\end{proof}

Above, we have shown that the direction $-\gk$ is \korange{a} common descent direction. It has been called  negative multi-gradient in \cite{StochMO24Vicente}. 
The problem \eqref{P2} as well as the reformulation \eqref{P2'} have   been studied in \cite{Fliege00} for calculating a steepest descent direction. As noted there in Lemma 1 this problem can also be used to characterize Pareto criticality: a point $x$ is Pareto critical if and only if the optimal value of \eqref{P2} is zero. Otherwise, the optimal value is negative. It has also been stated that there are other possibilities to determine  a so-called search direction, and the authors in \cite{Fliege00} propose to solve a variant of \eqref{P1} w.r.t.\ the norm $\|\cdot\|_\infty$.
With Lemma \ref{lemma:P1andP3}  we have  shown that, next to scaling, the directions obtained from \eqref{P2} and \eqref{P1} can be considered to be the same. 

Finally, we relate the negative of the objective function  of \eqref{P1} evaluated at $x=x^k$ in  $d=-\gk$ with $\omega(x^k)$ in the following lemma.

\begin{lemma}\label{lemma:key}
 Let  $x^k\in\R^n$ be any iteration point. Then  
\begin{equation}\label{eq:key}
    \|\gk\|_2^2  = -  \max_{j\in [m]}\nabla f_j(x^k)^\top (-\gk) 
\end{equation}
holds.
\end{lemma}
\begin{proof}
First, assume that $x^k$ is a Pareto critical point for \eqref{P}. Then  we have $\omega(x^k)=0$ and $\gk=0$. Then \eqref{eq:key} trivially holds. 
Next, we assume  that  $x^k$ is not  a Pareto critical point for \eqref{P} and thus $\omega(x^k)>0$. 
Then, by Lemma \ref{lemma:P1andP3}(b), we get
$\|g_k^s\|_2^2=-2h(-g_k^s)$ where the function $h$ is defined in (\ref{eq:defh}). 
By a simple calculation we derive \eqref{eq:key}.
\end{proof}

\subsection{Convergence rate analysis}\label{sec:convergence}

Recall that by \eqref{eq:Phi} we have  $\Phi(x) = \max_{j\in [m]} f_j(x)$ and that  $L_{\max}$  was defined   in \eqref{eq:Lmax}.
We start by stating a result on the descent in each iteration w.r.t.\ the values of the function $\Phi$.

\begin{lemma}\label{LemmaPhi}
Suppose that
Algorithm \ref{alg:MOADG} is applied to  \eqref{P}. Then, for all $k\geq 0$ it holds
\begin{equation}\label{eq:lemma21-1}
  \Phi(x^{k+1}) \leq \Phi(x^k) - \frac{\|g_k^s\|_2^2}{w_{k}} + \frac{L_{\max}}{2}  \|s^k\|_2^2,
\end{equation}
and
\begin{equation}\label{eq:lemma21-2}
  \Phi(x^0) - \Phi(x^{k+1}) \geq   \sum_{\ell = 0}^k  \frac{\|g^s_\ell\|_2^2}{w_\ell}  - \frac{L_{\max}}{2}\sum_{\ell=0}^k  \frac{\|g^s_\ell\|_2^2}{w_\ell^2}.
\end{equation}
\end{lemma}
\begin{proof} 
Recall that by Lemma \ref{lemma:key} we have  $ -\|\gk\|_2^2  =   \max_{j\in [m]}\nabla f_j(x^k)^\top (-\gk)$.
Considering that 
\[x^{k+1} - x^k = s^k = - \frac{1}{w_k}g^s_k,\] by the mean value theorem applied to each function $f_j,\, j\in [m]$, we have:

\[
\begin{array}{rl}
  \Phi(x^{k+1}) \leq &  \max_{j\in [m]} \left\{f_j(x^k) + \nabla f_j (x^k)^T s^k + \frac{L_j}{2} \|s^k\|_2^2 \right\}\\ [1.3ex]
     \leq  &  \displaystyle \Phi(x^k) + \frac{1}{w_{k}}  \max_{j\in [m]}\nabla f_j(x^k)^\top (-g_k^s) + \frac{L_{\max}}{2}  \|s^k\|_2^2\\ [1.5ex]
     =  &  \displaystyle  \Phi(x^k) - \frac{\|g_k^s\|_2^2}{w_{k}} + \frac{L_{\max}}{2} \|s^k\|_2^2,
\end{array}
\]
that is \eqref{eq:lemma21-1}.
Summing for $\ell = 0,\ldots,k$  gives 
\[
\begin{array}{rl}
  \Phi(x^{k+1}) \leq & \displaystyle  \Phi(x^0) - \sum_{\ell=0}^k \frac{\|g_\ell^s\|_2^2}{w_\ell} + \frac{L_{\max}}{2} \sum_{\ell=0}^k  \frac{\|g^s_\ell\|_2^2}{w_\ell^2}, \\ [1.5ex]
\end{array}
\]
that is inequality \eqref{eq:lemma21-2}. 
\end{proof}

Lemma \ref{LemmaPhi} gives a typical descent result when using Adagrad-like stepsizes, see e.g. \cite{gratton2024complexity,gratton2024bounds,prunAdag}: when the second term in the right hand-side of inequality (\ref{eq:lemma21-2}) is larger than the first term, the function $\Phi$ may increase from one iteration to the subsequent one. On the other hand, one expects that asymptotically the first term (linear in the weights) gets larger than the second (quadratic in the weights) and therefore to obtain a decrease in the function $\Phi$. This is exactly what happens when choosing weights as done in the Adagrad methods and as in our proposed algorithm: the weights are increasing during the iterations and then asymptotically the linear term becomes dominant.

\begin{remark}
Note that by Lemma \ref{lemma:P1andP3}(c), as $w_k>0$, 
the direction $s^k$ is a common descent direction and thus a descent direction for $\Phi$, i.e.,
 there exists $t_0>0$ such that it holds 
$f_j(x^k+t'\,s^k)<f_j(x^k)$    for all $ j\in[m]$,  
and, as a direct implication, also 
\[\Phi(x^k+t'\,s^k)<\Phi(x^k) \] 
for all $t'\in(0,t_0]$.
\end{remark}

Thanks to Lemma \ref{lemma:philow} we have that a value $\Phi_{low}$ exists such that 
\[
\Phi(x) \geq \Phi_{low},\; \mbox{ for all } x\in \R^n.
\]
We are now ready to state the announced result on the global convergence rate of Algorithm \ref{alg:MOADG}. For that, the average of a sequence of values $\beta^{\ell}\in\R$ is defined as 
\[\av_{\ell = 0,\ldots, k}\beta^{\ell}=\frac{1}{k+1}\sum\limits_{\ell=0}^k\beta^{\ell}.\]
Also, the proof uses the technical results of Lemmas \ref{log_bound} and \ref{Lemma:magic}  which are reported in Appendix~\ref{App:proof}.

\begin{theorem}\label{thm:rate}
Let $\Gamma_0 := \Phi(x^0) - \Phi_{\rm low}.$ 
Suppose that Algorithm \ref{alg:MOADG} is applied to  \eqref{P}.
 Then the following bound holds:
\[ 
\av_{\ell = 0,\ldots, k}\; \omega(x^\ell) = \av_{\ell = 0,\ldots, k} \; \|g_\ell^s\|_2^2 \leq \frac{\theta}{k+1}
\]
with 
  \[\theta = \max\left\{ \varsigma,  \frac \varsigma 2 e^{\frac{2\Gamma_0}{L_{\max}}},  
  \frac{2048 {L_{\max}}^4}{\varsigma} \right\}  \] 
and $\varsigma\in(0,1)$, as selected in Step 1 of  Algorithm \ref{alg:MOADG}.
\end{theorem}
\begin{proof}
Recall that by \eqref{eq:optvrelation} we have $  \omega(x^k)=\|\gk\|_2^2$ and we have further $\Gamma_0=\Phi(x^0)-\Phi_{\rm low}\geq \Phi(x^0)-\Phi(x^{k+1})$ for all $k$.
From \eqref{eq:lemma21-2},  we derive
\begin{equation}\label{eq:help1}
   \sum_{\ell=0}^k \frac{\|g^s_\ell\|_2^2}{w_\ell} \leq \Gamma_0 + \frac{L_{\max}}{2}  \sum_{\ell=0}^k \frac{\|g^s_{\ell}\|_2^2}{w_\ell^2}.
\end{equation}
 Using Lemma \ref{Lemma:magic} with $c_j = \|g_j^s\|_2^2$ and the fact that $w_k = \sqrt{\varsigma + \sum_{j=0}^k \|g_j^s\|_2^2}$ we have
 \[
 \sum_{\ell=0}^k \frac{\|g^s_{\ell}\|_2^2}{w_\ell^2} =  \sum_{\ell=0}^k \frac{\|g^s_{\ell}\|_2^2}{\varsigma + \sum_{j=0}^{\ell} \|g_j^s\|_2^2} \leq \log\left(1 + \frac 1 \varsigma \sum_{\ell=0}^k \|g^s_{\ell}\|_2^2\right),
\]
where $\log$ denotes the natural logarithmic function.
Therefore, including the above inequalities in \eqref{eq:help1} we get
\begin{equation}\label{eq:ineq0}
   \sum_{\ell=0}^k \frac{\|g^s_\ell\|_2^2}{w_\ell} \leq \Gamma_0 + \frac{L_{\max}}{2} \log\left(1 + \frac 1 \varsigma \sum_{\ell=0}^k \|g^s_{\ell}\|_2^2\right).
\end{equation}
Assume now that
\begin{equation}\label{ass:varsigma}
\sum_{\ell=0}^k \|g^s_\ell\|_2^2 \geq \max\left\{\varsigma, \frac \varsigma 2 e^{\frac{2\Gamma_0}{L_{\max}}}\right\},
\end{equation}
which implies  
\begin{equation}\label{eq:ineq1}
    \Gamma_0 \leq \frac{L_{\max}}{2} \log\left( \frac 2 \varsigma \sum_{\ell=0}^k \|g^s_{\ell}\|_2^2\right),
\end{equation}
as well as 
$$ \varsigma +  \sum_{\ell=0}^k \|g^s_{\ell}\|_2^2 \le 2 \sum_{\ell=0}^k \|g^s_{\ell}\|_2^2,$$
which leads to 
\begin{equation}\label{eq:help2}
 \log\left(1 + \frac 1 \varsigma \sum_{\ell=0}^k \|g^s_{\ell}\|_2^2\right)\leq \log\left( \frac 2 \varsigma \sum_{\ell=0}^k \|g^s_{\ell}\|_2^2\right).
\end{equation}
Furthermore, recalling the definition of $w_k$ in \eqref{weig}, we have that 
\begin{equation}\label{eq:ineq2}
w_k  = \sqrt{\varsigma + \sum_{\ell=0}^k \|g^s_{\ell}\|_2^2}\le \sqrt{2 \sum_{\ell=0}^k \|g^s_{\ell}\|_2^2}.
\end{equation}

Then, observing that $w_\ell \le w_k$ for all $\ell\le k$ and 
using inequalities (\ref{eq:ineq1}) and (\ref{eq:ineq2})  in (\ref{eq:ineq0})  using \eqref{eq:help2} we obtain
\[
\frac{\sum_{\ell=0}^k \|g^s_{\ell}\|_2^2}{\sqrt{2\sum_{\ell=0}^k \|g^s_{\ell}\|_2^2}}\leq L_{\max} \log\left( \frac 2 \varsigma \sum_{\ell=0}^k \|g^s_{\ell}\|_2^2\right)\korange{,}
\]
that is 
\begin{equation}\label{eq:sqrt}
    \frac{\sqrt{\varsigma}}{2} \sqrt{ \frac{2}{\varsigma} \sum_{\ell=0}^k \|g^s_{\ell}\|_2^2}
\leq 2 L_{\max} \log\left( \sqrt{\frac 2 \varsigma \sum_{\ell=0}^k \|g^s_{\ell}\|_2^2}\right).
\end{equation}
We can now apply Lemma \ref{log_bound} since (\ref{eq:sqrt}) is identical to (\ref{disu}) with 
\[
    a = \frac{\sqrt{\varsigma}}{2} , \quad b = 0,  \quad c = 2 L_{\max}, \quad u = \sqrt{\frac 2 \varsigma \sum_{\ell=0}^k \|g^s_{\ell}\|_2^2},
\]
and $u\ge 1$ due to  \eqref{ass:varsigma}, and obtain the bound
\[
\sum_{\ell=0}^k \|g^s_{\ell}\|_2^2  \le \frac{\varsigma}{2}  \max \left \{1,   \frac{64^2 {L_{\max}}^4}{\varsigma^2}         \right\}.
\]
Hence taking the average gives that
\[
\av_{\ell \in\{ 0,\ldots, k\}} \|g^s_{\ell}\|_2^2 \leq   \max \left \{\frac{\varsigma}{2},   
\frac{2048 {L_{\max}}^4}{\varsigma}         \right\} \frac{1}{k+1}.
\]
Suppose that \eqref{ass:varsigma} fails, then we would have
\[
\av_{\ell \in\{ 0,\ldots, k\}} \|g^s_{\ell}\|_2^2 \leq   \max\left\{\varsigma, \frac \varsigma 2 e^{\frac{2\Gamma_0}{L_{\max}}}\right\} \frac{1}{k+1}.
\]
\end{proof}

Theorem \ref{thm:rate} proves that the average $\omega(x^k)$ converges to zero.  
Furthermore, as
\[\min_{\ell \in\{ 0,\ldots, k\}} \|g_{\ell}^s\|_2\leq \sqrt{\av_{\ell \in\{ 0,\ldots, k\}} \|g^s_{\ell}\|_2^2}\leq \sqrt{\frac{\theta}{k+1}},\] 
we get that the global convergence rate of the average norm of the common descent direction $g^s_k$ is $\mathcal{O}\left( 1/\sqrt{1+k}\right)$, in accordance with the original Adagrad algorithm.

\section{Numerical results}\label{sec:num}
In this section, we present the numerical tests conducted to evaluate the performance of~\texttt{MO-Adagrad}.
We implemented~\texttt{MO-Adagrad} described in Algorithm \ref{alg:MOADG} in MATLAB, and we set $\varsigma = 10^{-2}$ (see e.g. \cite{gratton2024bounds,prunAdag}). 

For comparison, we implemented a standard descent method in which $-g_k^s$ (obtained solving~\eqref{P3} for $x=x^k$) serves as common descent direction and the step length is determined by using an Armijo-like rule as proposed in~\cite{Fliege00}. 
The method is referred to as~\texttt{MO-Descent} in the forthcoming sections. 
We report the scheme of~\texttt{MO-Descent} in Algorithm~\ref{alg:MODesc} for completeness.
In our implementation of~\texttt{MO-Descent}, we set $\beta = 0.1$.

All tests have been run on an Intel$^{\textnormal{®}}$ Core™ i5-14400F processor running at 2.80GHz under Linux. 

\begin{algorithm}{}
  \caption{\texttt{MO-Descent}}\label{alg:MODesc} 
  	\begin{algorithmic}[1]
    \State Initialization: a starting point $x^0 \in \R^n$ and a constant $\beta\in (0,1)$ are given. Set $k=0$.    
    
  \State Solve \eqref{P3} for $x = x^k$ with optimal solution $\lambda^k$ to obtain $g^s_k\in \R^n$ as in \eqref{eq:defdescentd}
  \State Set $t = 1$.
  \While{$f_j(x^k -t g^s_k) > f_j(x^k) -\beta t \nabla f_j(x^k)^T g^s_k$ \quad  for some $j\in [m]$}
  \State $t = t/2$
  \EndWhile
  \State Set $\alpha^k = t$.
  \State   Define
   $$  \qquad x^{k+1} = x^k - \alpha^k g^s_k,$$ 
    increment $k$ by one and return to Step~2.
  \end{algorithmic}
  \end{algorithm}

To solve~\eqref{P3} we adopt \texttt{fmincon} within MATLAB. Note that the variable space of~\eqref{P3} has dimension $m$, meaning that computing $-g_k^s$ through~\eqref{P3} is computationally convenient when dealing with a small number of objective functions. Finally, for both methods, we say that  a Pareto critical point is reached when $\|g_k^s\| \le 10^{-3}$.

In our numerical experience section, we first provide an overview on the performance of ~\texttt{MO-Adagrad} on bi-objective unconstrained problems built from the CUTEst collection from~\cite{gratton2025s2mpj} available in MATLAB form at \url{https://github.com/GrattonToint/S2MPJ}, as well as on standard unconstrained multi-objective problems from the literature (see, e.g., \cite{Thomann19}). 
We also consider noisy problem instances to assess the behavior of \texttt{MO-Adagrad} under different noise levels.
Then, we report the results of preliminary numerical experiments on the use of \texttt{MO-Adagrad} for training multi-task problems. Indeed, we considered two synthetic datasets: one represents a 4-classes classification task and a binary classification task, the other one consists of 2
binary classification tasks, both for a set of points in $\R^2$, see, e.g., \cite{Multi-Task}. 

\subsection{Experiments on CUTEst instances.}
In a first set of experiments, we constructed bi-objective unconstrained instances from the single-objective unconstrained problems in the CUTEst collection~\cite{gratton2025s2mpj} by adding a squared $\ell_2$-norm regularizer as a second objective. We excluded all problems with optimal value equal to $0$ in order to obtain proper bi-objective instances. Overall, we built 124 instances; the names of used CUTEst problems  are reported in Table \ref{tab::cutestinst}  along with their dimension that goes up to $4999$\footnote{The problem dimensions are derived from the CUTEst test set available online at \url{https://github.com/GrattonToint/S2MPJ}. Although these dimensions may occasionally differ from those in the official CUTEst repository, we opted to use the 
MATLAB version.}.
All runs were initialized from the standard starting points provided by \cite{gratton2025s2mpj}.
We stopped the algorithms as soon as a Pareto critical point is reached or in case a limit of 100000 number of gradient 
evaluations of the form  $(\nabla f_1(x), \ldots,\nabla f_m(x))$ is reached. 
For a fair comparison, for~\texttt{MO-Descent} we counted the number of function evaluations $f(x)$ divided by the problem dimension $n$ and added this quantity to the number of gradient evaluations. 

We report in Figure~\ref{fig:cutest+2norm} the performance profiles~\cite{dolan2002benchmarking} related to the  number of gradient (and weighted function) evaluations. Clearly, no function evaluations are required by~\texttt{MO-Adagrad}. 
While~\texttt{MO-Descent} requires a smaller number of gradient and function evaluations in case a Pareto critical point is reached,~\texttt{MO-Adagrad} turns to be more robust solving the 89\% of instances (\texttt{MO-Descent} only solves the 77\% of instances).
\begin{figure}[t]
  \centering
  \caption{Performance profiles (log$_{10}$-scale) on the number of gradient (and function) evaluations on instances derived from CUTEst problems by including $\ell_2$-norm regularizer as a second objective.}
  \includegraphics[width=0.8\textwidth]{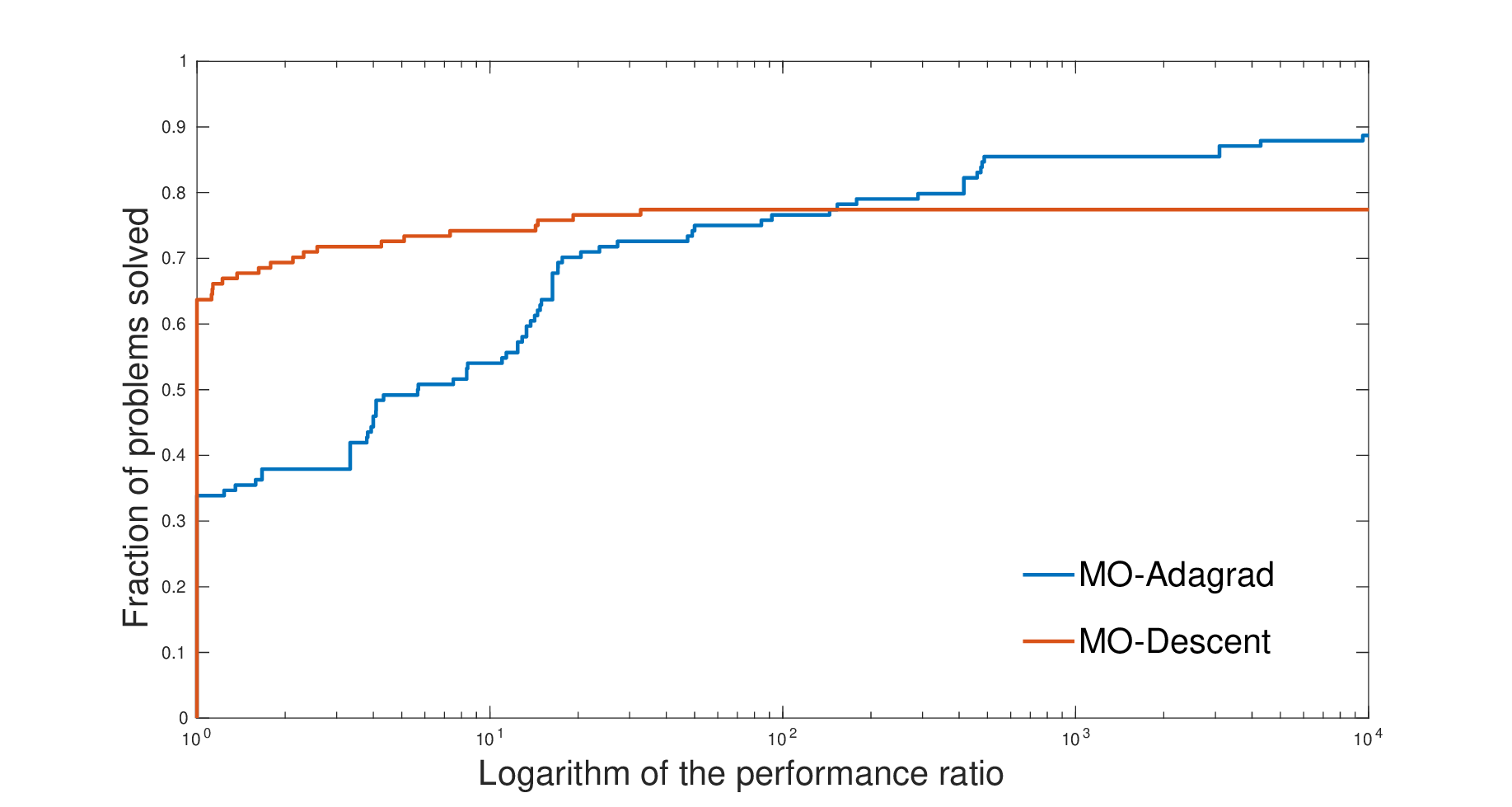}
      \label{fig:cutest+2norm}
\end{figure}

\begin{table}[h]
\centering
\scriptsize
\caption{CUTEst instances considered to build bi-objective problems by including squared $\ell_2$-norm regularizer as a second objective.}
\label{tab::cutestinst}
  \vspace{2mm}
\resizebox{1\textwidth}{!}{
\begin{tabular}{c l|c l|c l|c l}
\hline
{Problem} & $n$ &  {Problem} & $n$ &{Problem} & $n$ &{Problem} & $n$ \\
\hline
ALLINITU	&	4	&	EGGCRATE	&	2	&	LSC1LS	&	3	&	PENALTY2	&	10	\\
ARGLINA	&	200	&	EIGENALS	&	6	&	LSC2LS	&	3	&	POWELLBSLS	&	2	\\
ARGLINB	&	10	&	EIGENBLS	&	6	&	LUKSAN12LS	&	98	&	POWERSUM	&	10	\\
ARWHEAD	&	10	&	ELATVIDU	&	2	&	LUKSAN17LS	&	100	&	QING	&	5	\\
BARD	&	3	&	ENGVAL1	&	10	&	LUKSAN22LS	&	100	&	QUARTC	&	10	\\
BDQRTIC	&	10	&	ENGVAL2	&	3	&	MANCINO	&	10	&	ROSENBR	&	2	\\
BEALE	&	2	&	ERRINROS	&	10	&	MARATOSB	&	2	&	ROSSIMP1	&	2	\\
BIGGS6	&	6	&	ERRINRSM	&	10	&	METHANL8LS	&	31	&	ROSSIMP2	&	2	\\
BOXBODLS	&	2	&	EXPFIT	&	2	&	MEXHAT	&	2	&	ROSSIMP3	&	2	\\
BRKMCC	&	2	&	EXTROSNB	&	10	&	MEYER3	&	3	&	ROSZMAN1LS	&	4	\\
BROWNAL	&	10	&	FLETBV3M	&	10	&	MGH17LS	&	5	&	S308	&	2	\\
BROWNDEN	&	4	&	FLETCBV2	&	10	&	MGH17SLS	&	5	&	SENSORS	&	5	\\
BROYDN3DLS	&	5	&	FLETCBV3	&	10	&	MODBEALE	&	10	&	SINQUAD2	&	10	\\
CHNROSNB	&	5	&	FREUROTH	&	4	&	MSQRTALS	&	25	&	SINQUAD	&	10	\\
CHNRSNBM	&	5	&	GAUSSIAN	&	3	&	MSQRTBLS	&	25	&	SPMSRTLS	&	4999	\\
CLUSTERLS	&	2	&	GENROSE	&	10	&	n10FOLDTRLS	&	4	&	STREG	&	4	\\
COOLHANSLS	&	9	&	GROWTHLS	&	3	&	NONCVXU2	&	10	&	THURBERLS	&	7	\\
CRAGGLVY	&	10	&	GULF	&	3	&	NONCVXUN	&	10	&	TOINTGOR	&	50	\\
CUBE	&	2	&	HATFLDD	&	3	&	NONDIA	&	10	&	TOINTGSS	&	10	\\
CURLY10	&	15	&	HATFLDE	&	3	&	OSBORNEA	&	5	&	TOINTQOR	&	50	\\
CYCLIC3LS	&	12	&	HIMMELBCLS	&	2	&	OSBORNEB	&	11	&	TQUARTIC	&	10	\\
CYCLOOCFLS	&	20	&	HIMMELBH	&	2	&	OSCIGRAD	&	10	&	TRIDIA	&	5	\\
DANIWOODLS	&	2	&	HYDCAR6LS	&	29	&	PALMER1C	&	8	&	TRIGON2	&	10	\\
DENSCHNB	&	2	&	INDEFM	&	10	&	PALMER1D	&	7	&	VARDIM	&	10	\\
DENSCHNC	&	2	&	JENSMP	&	2	&	PALMER2C	&	8	&	WATSON	&	12	\\
DENSCHND	&	3	&	JUDGE	&	2	&	PALMER3C	&	8	&	WAYSEA1	&	2	\\
DENSCHNF	&	2	&	KSSLS	&	10	&	PALMER4C	&	8	&	WOODS	&	4000	\\
DEVGLA2	&	5	&	LANCZOS1LS	&	6	&	PALMER5C	&	6	&	YATP2CLS	&	35	\\
DIXON3DQ	&	10	&	LANCZOS2LS	&	6	&	PALMER6C	&	8	&	YATP2LS	&	35	\\
EDENSCH	&	10	&	LANCZOS3LS	&	6	&	PALMER7C	&	8	&	YFITU	&	3	\\
EG2	&	10	&	LIARWHD	&	10	&	PALMER8C	&	8	&	ZANGWIL2	&	2	\\
\hline
\end{tabular}}
\end{table}

\subsection{Experiments on noisy problems}
In a second set of experiments, we selected the unconstrained problems used in \cite{Thomann19} as benchmark multi-objective optimization instances. Moreover, we
 constructed $10$ bi-objective unconstrained instances combining single-objective unconstrained problems in the CUTEst collection~\cite{gratton2025s2mpj} sharing the same dimension.
 In this case, we considered three different starting points: the starting point for the first original problem, the starting point for the second original problem and the average of the two.
For the benchmark problems, we used $10$ randomly generated starting points.
The results in the forthcoming table are reported on average.

As before, we stopped the algorithms as soon as a Pareto critical point is reached or in case the limit of 100000 number of gradient (and function in case of~\texttt{MO-Descent}) evaluations is reached. 
As the objective-function-free methods are mainly known for their ability of being robust to noise, we evaluated our approach under different levels of noise added both to the objectives and to the gradients. Specifically, we considered a noise level of $\rho \in \{5\%, 15\%, 25\%\}$  by adding a relative Gaussian noise with unit variance (see, e.g., \cite{gratton2024bounds}).
We perturbed each element function as follows
$$\tilde f_j(x) = f_j(x)(1+ \rho \mathcal{N}(0,1)), \qquad  j\in [m]$$
 where $\mathcal{N}(0,1)$ is the standard normal distribution. Gradients are perturbed componentwise, analogously. 

We report in Table~\ref{tab::mix-cutest+bench} the results obtained in terms of gradient and function evaluations.  
Consistent with our first set of experiments, in the noise-free setting, \texttt{MO-Adagrad} exhibits fewer failures than \texttt{MO-Descent}, the latter of which fails on 20 instances.  
When noise is added to both the gradients and the objective functions, \texttt{MO-Descent} shows improved performance, whereas \texttt{MO-Adagrad} remains largely unchanged\footnote{We also ran experiments with $\rho=10^{-3}$ and $\rho=10^{-4}$. While the results are close to those of the noise-free setting, a small amount of noise allowed \texttt{MO-Descent} to recover from some failures, consistent with our observations for $\rho = 0.05$.}. In particular, \texttt{MO-Descent} considerably reduces its number of failures  for all the considered values of $\rho$  (both \texttt{MO-Descent}  and \texttt{MO-Adagrad} solve the 5 benchmark problems for the 10 starting points and the different noise levels within the limit of gradient and function evaluations).

\begin{table}[h!]
\centering
\scriptsize
\caption{Results on the instances obtained combining CUTEst problems and on the benchmark instances. For the CUTEst problems, three starting points are considered: the starting point for the first original problem, the starting point for the second original problem and the average of the starting points of the two original problems. For the benchmark instances, we use $10$ different randomly generated starting points. All the results are averaged in gradient (and function) evaluations over the number of instances solved (reported in brackets).}
\label{tab::mix-cutest+bench}
  \vspace{2mm}
\resizebox{1\textwidth}{!}{
\begin{tabular}{c|c|c|c|c|c}
\hline
{Problem} & $n$ & \multicolumn{2}{c|}{no noise} & \multicolumn{2}{c}{$\rho = 5\%$} \\
&  & \texttt{MO-Descent} & \texttt{MO-Adagrad} & \texttt{MO-Descent} & \texttt{MO-Adagrad}  \\
\hline
BROWNDEN-ALLINITU & 	 4 	&	 -  	&	{\bf 13(1)}	&	5484.3(1)	&	 {\bf 14(1)} \\
BROWNAL-ARWHEAD  &	 10 	&	4.9(2)	&	{\bf 25.7(3)}	&	3237.5(3)	&	{\bf 866.7(3)} \\
BROWNAL-VARDIM  &	 10 	&	{\bf 657.4(3)}	&	11830.7(3)	&	{\bf 165.9(3)}	&	10253.3(3) \\
ARWHEAD-VARDIM & 	 10 	&	78.0(1)	&	{\bf 6.7(3)}	&	1039.4(2)	&	{\bf 126.3(3)} \\
ZANGWIL2-ROSENBR & 	 2 	&	-	&	{\bf 64.0(3)}	&	2729.2(3)	&	{\bf 219.0(3)} \\
ZANGWIL2-CUBE & 	 2 	&	56.0(1)	&	{\bf 94.0(3)}	&	700.2(3)	&	{\bf 99.0(3)} \\
ZANGWIL2-WAYSEA1 & 	 2 	&	-	&	{\bf 217.7(3)}	&	1881.3(3)	&	{\bf 173.7(3)} \\
ROSENBR-WAYSEA1 & 	 2 	&	-	&	{\bf 4887.3(3)}	&	22722.7(3)	&	{\bf 5156.0(3)} \\
ROSENBR-CUBE & 	 2 	&	1215.5(3)	&	{\bf 581(3)}	&	1107.7(3)	&	{\bf 772.0(3)} \\
WAYSEA1-CUBE & 	 2 	&	-	&	{\bf 719(2)}	&	9594.3(2)	&	{\bf 334.5(2)} \\
Lovison 3 & 2 &  {\bf 24.2(10)} & 1685.8(10) & 1995.6(10) & {\bf 1115.2(10)}\\ 
Lovison 4 & 2 &  {\bf 15.6(10)} & 3422.4(10) & {\bf 270.2(10)} & 3356.2(10)\\ 
MOP1 & 2 & {\bf 3.0(10)} & 2536.7(10) & {\bf 17.1(10)} & 2517.1(10)\\ 
T1 & 2 &  {\bf 31.8(10)} & 3844.7(10) & {\bf 837.0(10)} & 1944.8(10)\\ 
T2 & 2 &  12.2(10) & {\bf 7.1(10)} & 58.1(10) & {\bf 7.2(10)}\\ 
\hline
\hline
{Problem} & $n$ & \multicolumn{2}{c|}{$\rho = 15\%$} & \multicolumn{2}{c}{$\rho = 25\%$} \\
&  & \texttt{MO-Descent} & \texttt{MO-Adagrad} & \texttt{MO-Descent} & \texttt{MO-Adagrad}  \\
\hline
 BROWNDEN-ALLINITU & 	 4 	&	{\bf 23526.9(2)}	&	16(1)	&	8693.75(1)	&	 {\bf 19(1)}\\
 BROWNAL-ARWHEAD & 	 10 	&	680.5(3)	&	{\bf 135.3(3)}	&	259.2(3)	&	{\bf 57.7(3)} \\
 BROWNAL-VARDIM & 	 10 	&	{\bf 746.7(3)}	&	7576.3(3)	&	{\bf 338.9(3)}	&	4570.3(3) \\
 ARWHEAD-VARDIM & 	 10 	&	663.3(3)	&	{\bf 48.7(3)}	&	421.2(3)	&	{\bf 25.3(3)} \\
 ZANGWIL2-ROSENBR & 	 2 	&	6036.8(3)	&	{\bf 362.3(3)}	&	4361.3(3)	&	{\bf 214.0(3)} \\
 ZANGWIL2-CUBE & 	 2 	&	1126.0(3)	&	{\bf 381.7(3)}	&	2302.2(3)	&	{\bf 432.7(3)} \\
 ZANGWIL2-WAYSEA1 & 	 2 	&	737.7(3)	&	{\bf 93.7(3)}	&	559.0(3)	&	{\bf 79.0(3)} \\
 ROSENBR-WAYSEA1 & 	 2 	&	53622.8(3)	&	{\bf 6964.0(3)}	&	37562.0(3)	&	{\bf 5334.0(3)} \\
 ROSENBR-CUBE & 	 2 	&	1377.0(3)	&	{\bf 696.3(3)}	&	{\bf 861.3(3)}	&	2365.0(3) \\
 WAYSEA1-CUBE & 	 2 	&	4939.0(2)	&	{\bf 123(2)}	&	7984.0(2)	&	{\bf 679.5(2)} \\
 Lovison 3 & 2 & 9847.5(10) & {\bf 1265.9(10)} & 23562.0(10) & {\bf 1506.5(10)}\\ 
Lovison 4 & 2 & \textbf{789.0(10)} & 3190.6(10) & {\bf 1778.3(10)} & 2875.0(10) \\
 MOP1 & 2 & {\bf 4.5(10)} & 2401.6(10) & {\bf 4.5(10)} & 2214.6(10)\\ 
 T1 & 2 & 2627.6(10) & {\bf 2464.1(10)} & {\bf 1443.5(10)} & 2093.6(10)\\ 
 T2 & 2 & 65.1(10) & {\bf 7.2(10)}  & 68.2(10) & {\bf 6.6(10)}\\ 
\hline
\end{tabular}}
\end{table}

At first glance, this suggests that the presence of noise helps \texttt{MO-Descent} in identifying Pareto critical points. To investigate this further, we examined the locations of the Pareto critical points in the criterion space.
We report in Table~\ref{tab::mix-cutest-fo} the distance between the Pareto critical points detected in the presence of noise and those obtained in the noise-free setting (on average on the runs sharing the same starting points).\footnote{Note that, for ARWHEAD-VARDIM, \texttt{MO-Descent} with $\rho = 5\%$ identified Pareto critical points from starting points different from the one that led to a solution in the noise-free setting;  consequently, the distance could not be computed.}
The results show that the points found by \texttt{MO-Descent} exhibit larger distances compared to those obtained by \texttt{MO-Adagrad} for the majority of the instances.
While this indicates that noise can drive exploration into different regions of the criterion space, it may also represent an undesirable behavior: \texttt{MO-Descent} may fail to reliably converge to the same solutions as in the noiseless setting, whereas \texttt{MO-Adagrad} demonstrates more robust convergence, highlighting its reliability in noisy scenarios.

\begin{table}
\centering
\scriptsize
\caption{Distance from the Pareto critical points detected with no noise.}
\label{tab::mix-cutest-fo}
  \vspace{2mm}
\resizebox{1\textwidth}{!}{
\begin{tabular}{c|c c| c c|c c}
\hline
Problem & \texttt{MO-Descent} & \texttt{MO-Adagrad} & \texttt{MO-Descent} & \texttt{MO-Adagrad} & \texttt{MO-Descent} & \texttt{MO-Adagrad} \\
&  \multicolumn{2}{c|}{$\rho = 5\%$} & \multicolumn{2}{c|}{$\rho = 15\%$} & \multicolumn{2}{c}{$\rho = 25\%$}\\
\hline
BROWNAL-ARWHEAD 	& 4.57 & {\bf 2.28} & 5.13 & {\bf 2.48} & 7.25 & {\bf 2.64}\\ 
BROWNAL-VARDIM 	& {\bf 0.01} & 0.02 & {\bf 0.01} & 0.04 & {\bf 0.01} & 0.06\\ 
ARWHEAD-VARDIM & - & {\bf 564.17} & 10340.18 & {\bf 225.24} & 10139.62 & {\bf 1343.01}\\
ZANGWIL2-CUBE & 176.16 & {\bf 27.15} & {\bf 1.71} & 81691.15 & {\bf 10.25} & 90851.88 \\
ROSENBR-CUBE & 0.28 & {\bf 0.05} & {\bf $<$ 1e-2} & 0.13 & {\bf $<$ 1e-2} & 0.33\\ 
Lovison 3 & 100.37 & {\bf 32.58}& 142.87 & {\bf 113.09}& 412.79 & {\bf 210.07}\\ 
Lovison 4 & 15.41 & {\bf 1.45}& 27.22 & {\bf 4.09}& 25.21 & {\bf 8.81}\\ 
MOP1 & 0.25 & {\bf 0.11}& 3.77 & {\bf 0.46}& 3.91 & {\bf 0.76}\\ 
T1 & 37.56 & {\bf 8.03}& 59.61 & {\bf 22.38}& 82.10 & {\bf 35.01}\\ 
T2 & {\bf 0.05} & 0.06& {\bf 0.16} & 0.17& 0.28 & 0.28\\ 
\hline
\end{tabular}}
\end{table}

\subsection{Experiments on multi-task geometric classification}
We consider a data set $\mathcal{D} = \{(\mathbf{x}_i, y_i^{(1)}, y_i^{(2)})\}_{i=1}^N$ consisting of $N = 10000$ points uniformly sampled in the square $[-2,2]^2 \subseteq \R^2$. In the first example, called {\em Quadrants-Circle}, points are labeled with respect to two different criteria: 
Task 1 to belong to one of the four quadrants of $\R^2$ (4-classes classification task); Task 2 to be inside or outside the unit  circle centered in the origin (binary classification task).

Since Task 2 (Circle) is not linearly separable in the original 2D space, the input vector $\mathbf{x}_i$ is obtained via a feature engineering \cite{bishop2006pattern} that includes quadratic terms:
\begin{equation*}
    \mathbf{x}_i = [1, x_{i,1}, x_{i,2}, x_{i,1}^2, x_{i,2}^2]^\top  \in \mathbb{R}^5
\end{equation*}
where the first element $1$ represents the bias term,  $x_{i,1}$ and $x_{i,2}$ are the components of the point $x_i \in [-2,2]^2$, for $i = 1,\dots,N$.
The labels are defined as follows:
\begin{itemize}
    \item $y_i^{(1)} \in \{1, 2, 3, 4\}$: Quadrant Label.
    \item $y_i^{(2)} \in \{0, 1\}$: Circle Label ($1$ if inside, $0$ if outside).
\end{itemize}
As prediction models, we considered the Softmax Regression for Task 1 and the  Logistic Regression for Task 2. Finally, we employed categorical Cross-Entropy 
and the binary Cross-Entropy as loss functions for the two tasks.

Overall, the final bi-objective problem is:
\begin{equation*}
    \min_{\mathbf{W}_1, \mathbf{w}_2}\ (J_1(\mathbf{W}_1), J_2(\mathbf{w}_2))
\end{equation*}
where $\mathbf{W}_1 \in \mathbb{R}^{5\times 4}$ and $\mathbf{w}_2 \in \mathbb{R}^{5}$ denote the weights of Task 1 and Task 2, respectively. The 
(non-competing) objectives are
\begin{equation*}
    J_1(\mathbf{W}_1) = -\frac{1}{N} \sum_{i=1}^{N} \sum_{k=1}^{4} \mathds{1}(y_i^{(1)} = k) \log \left( \frac{\exp(\mathbf{w}_{1,k}^\top \mathbf{x}_i)}{\sum_{j=1}^4 \exp(\mathbf{w}_{1,j}^\top \mathbf{x}_i)} \right),
\end{equation*}
and 
\begin{equation}\label{loss_bin}
    J_2(\mathbf{w}_2) = -\frac{1}{N} \sum_{i=1}^{N} \left[ y_i^{(2)} \log(\hat{y}_i^{(2)}) + (1 - y_i^{(2)}) \log(1 - \hat{y}_i^{(2)}) \right],
\end{equation}
where $\mathds{1}$  denotes the indicator function, $\hat{y}_i^{(2)} = \sigma(\mathbf{w}_2^\top \mathbf{x}_i)$ with $\sigma$ the sigmoid function and $\mathbf{w}_{1,j}$ denotes the $j$th column of $\mathbf{W}_1$, for $j =1,\dots,4$.

In the second example, named {\em Diagonals-Circle}, Task 1 is a non-linear binary classification task based on the XOR logic of the Cartesian quadrants. The model must distinguish between points in the main diagonal (Quadrants 1 and 3) and the anti-diagonal (Quadrants 2 and 4). This requires the inclusion of an interaction feature ($x \cdot y$) to achieve linear separability. In this case, we used the objective function $J_2$ in (\ref{loss_bin}) for both objectives.

In both examples, the data set is randomly split into a training set  and a  test set made up of 8000 and 2000 samples, respectively. A representation of the predicted labels on the testing set for the two multi-task problems  obtained with \texttt{MO-Adagrad} is reported in Figures \ref{fig:MTquad} and \ref{fig:MTbin}.

\begin{figure}[h!]
  \centering
  \caption{{\em Quadrants-Circle}: Task 1 on the left, Task 2 on the right. Different colors denote different labels for the points. }
  \includegraphics[width=\textwidth]{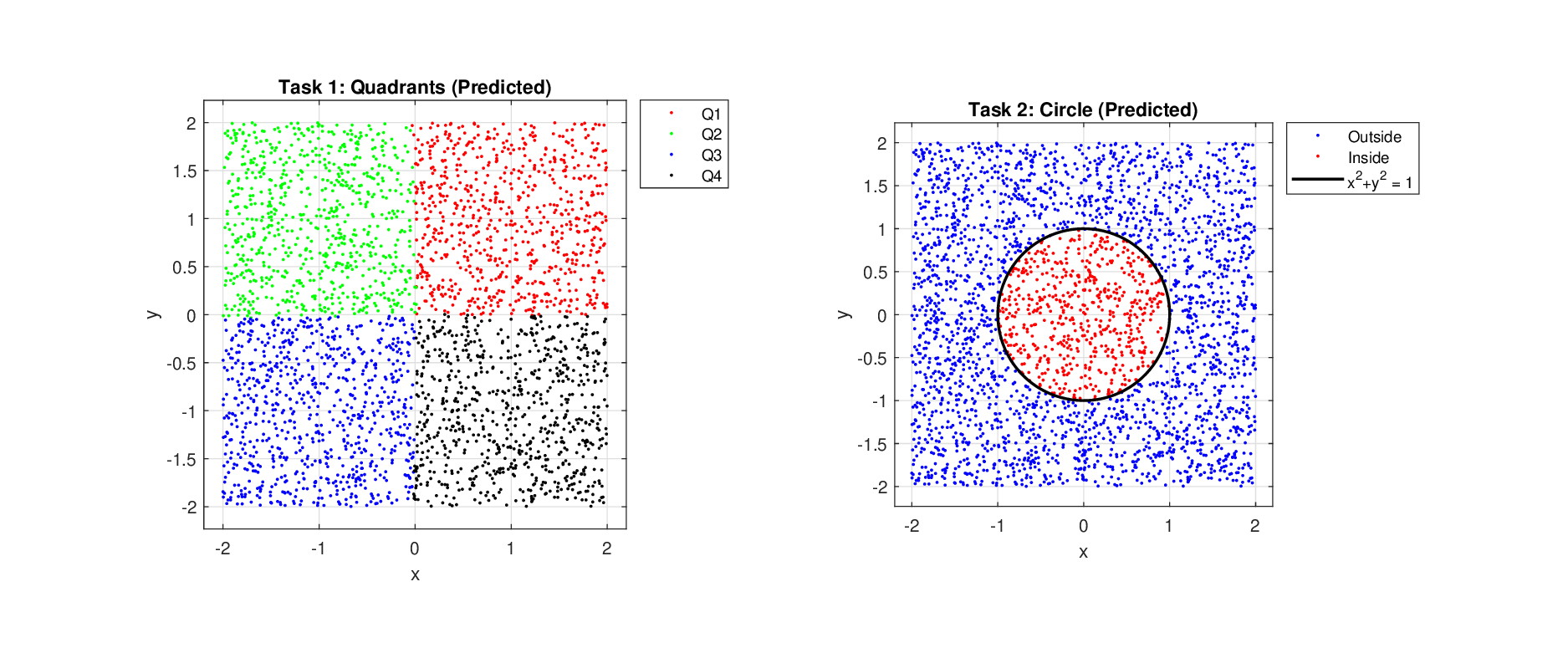}
      \label{fig:MTquad}
\end{figure}

\begin{figure}[h!]
  \centering
  \caption{{\em Diagonals-Circle}: Task 1 on the left, Task 2 on the right. Different colors denote different labels for the points. }
  \includegraphics[width=\textwidth]{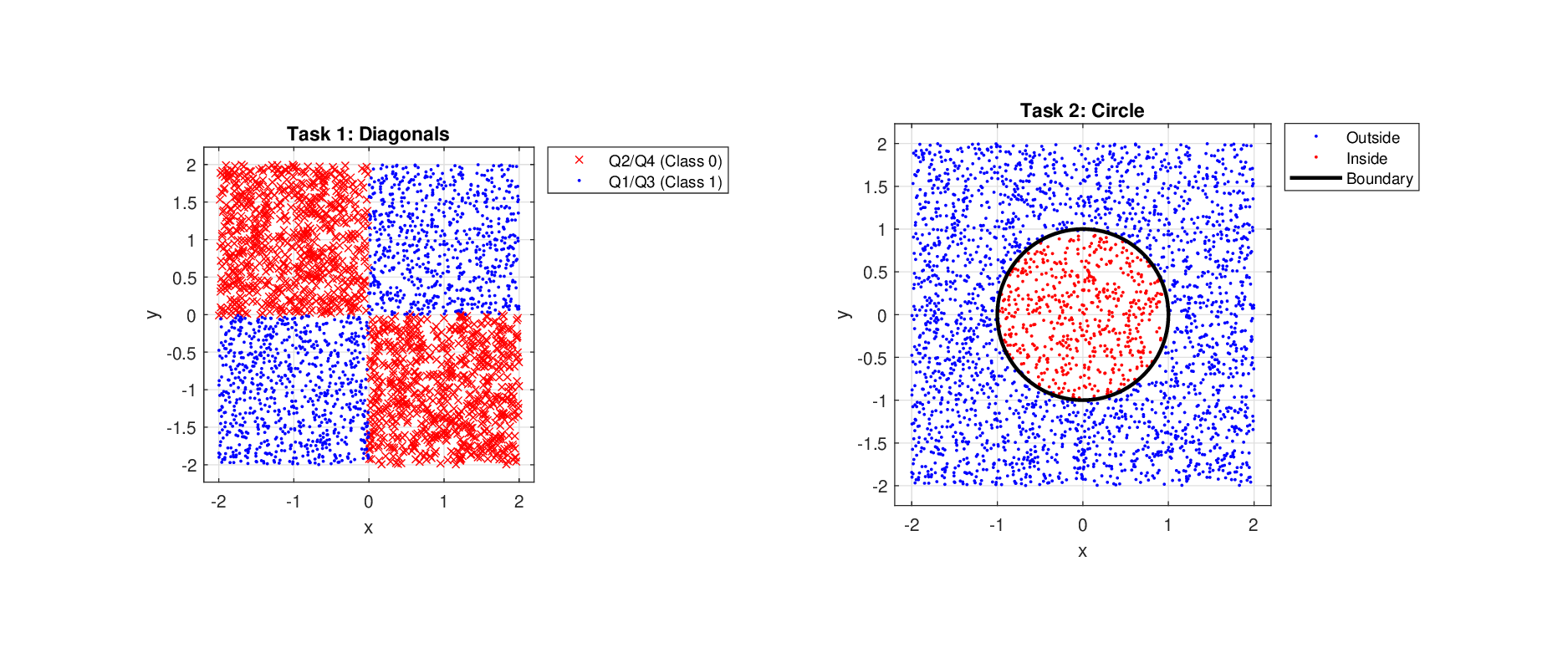}
      \label{fig:MTbin}
\end{figure}

In Table \ref{tab::MT} we report the results obtained with \texttt{MO-Descent} and \texttt{MO-Adagrad} on the two multi-task examples. In particular, we let the solvers run for 1000 iterations, and we retrieved the number of gradient/function evaluations ($g_{eval}/f_{eval}$) and elapsed time in seconds (time) employed in the training phase, to get the maximum accuracy $\mathcal{A}$ on the testing set. The accuracy $\mathcal{A}$ is measured considering the minimum of the testing accuracies in the two tasks. We observe that the two solvers get comparable
accuracies (slightly in favor of \texttt{MO-Descent}) but  \texttt{MO-Adagrad} is faster by a factor 3 than \texttt{MO-Descent}. This is due to the objective-function-free nature of \texttt{MO-Adagrad} that makes it particularly suitable when, as in this multi-task framework, evaluating the objective function is roughly as expensive as evaluating the gradient.

\begin{table}[h!]
\centering
\scriptsize
\caption{Results on the multi-task geometric classification instances.}
\label{tab::MT}
  \vspace{2mm}
\resizebox{1\textwidth}{!}{
\begin{tabular}{c |cccc |ccc }
\hline
\hline
\multirow{2}{*}{Problem} & \multicolumn{4}{c|}{\texttt{MO-Descent}} & \multicolumn{3}{c}{\texttt{MO-Adagrad}}\\
 &   $g_{eval}$ & $f_{eval}$ & time & $\mathcal{A}$ &   $g_{eval}$ & time &$\mathcal{A}$  \\
\hline
{\em Quadrants - Circle}  &  306 & 3957 & 24.9 & 99.8\% &  388  & {\bf 6.9} & 99.6\%  \\ 
{\em Diagonals - Circle}  &  577 & 6324 & 27.3 & 99.9\% &  970  & {\bf 10.0} & 99.1\%  \\ 
 \hline
 \hline
\end{tabular}}
\end{table}

\section{Conclusions}\label{sec:concl}
In this work, we propose a novel Adagrad-like algorithm, {\tt MO-Adagrad}, for unconstrained multi-objective optimization.
Our method does not need line search based techniques or dominance-based criteria for accepting new points, while still converging to Pareto critical points. Through extensive numerical tests on a broad set of bi-objective instances, we show the robustness of \texttt{MO-Adagrad} in comparison with a line search based first-order method. 
The method looks robust and promising to deal with noisy problems, as known for its single-objective counterpart.
Further preliminary experiments have been devoted to training multi-task problems, suggesting the potential ability of \texttt{MO-Adagrad} to deal with multi-task learning problems.

Future work will be devoted to extend \texttt{MO-Adagrad} to compute sets of nondominated Pareto points, as opposed to the single Pareto point currently produced. Recent advances in multi-objective optimization have seen growing interest in local algorithms that generate sequences of sets, with the aim of approximating the Pareto set of a multi-objective problem~\cite{cocchi2020convergence,custodio2011direct,lapucci2023improved,lapucci2023limited,
lapucci2026effective,liuzzi2016derivative}.
However, avoiding the use of dominance-based acceptance criteria presents a significant challenge when generating a sequence of sets, rather than a sequence of individual points. Consequently, future work will focus on developing alternative approaches, beyond multistart procedures, to tackle this complexity.
 
{\footnotesize
\subsection*{\footnotesize Acknowledgment}
The authors would like to thank the anonymous referee for their detailed and constructive comments that allowed us to improve the quality of our paper.

\subsection*{\footnotesize Ethics approval and consent to participate}
Not applicable.
\subsection*{\footnotesize Consent for publication}
Not applicable.
\subsection*{\footnotesize Funding}
{\footnotesize
The work of M.P. was partially supported by INdAM-GNCS under the
project CUP E53C24001950001.  The research of M.P.  was  partially granted by the Italian Ministry of University and Research (MUR) through the PRIN 2022 ``MOLE: Manifold constrained Optimization and LEarning'',  code: 2022ZK5ME7 MUR D.D. financing decree n. 20428 of November 6th, 2024 (CUP B53C24006410006). A research visit of G.E. at University of Florence with M.P. and M.d.S. was funded by the Erasmus+ Programme of the European Union.
}
\subsection*{\footnotesize Availability of data and materials}
The data presented in this manuscript are reproducible through the implementation publicly available on \\ \url{https://github.com/mariannadesantis/MO-ADAGRAD}.

\subsection*{\footnotesize Competing interests}
The authors declare that they have no competing interests.
\subsection*{\footnotesize Authors' contributions}
All authors contributed equally to the writing of this article. All authors reviewed the manuscript.}

\appendix
\section{Technical lemmas} \label{App:proof}

We report in this section two technical results. The proof of the first one can be found in~\cite{bellavia2026} and is reported for the sake of completeness. The proof of Lemma~\ref{Lemma:magic} can be found in~\cite[Lemma 3.1]{gratton2024complexity}.
In the following, with $\log$ we denote the natural logarithmic function. 

\begin{lemma}\label{log_bound}
Suppose that 
\begin{equation}\label{disu}
     a u \leq b + c \log(u),
\end{equation}
for some $a,c>0$, a scalar $b \in\R$ and $u\ge1$. Then
\begin{equation*}
    u \leq \max \left\{e^{b/c}, \frac{4c^2}{a^2} \right \}.
\end{equation*}
\end{lemma}
\begin{proof}
Suppose that $u \geq e^{b/c}$. Then $b\le c \log(u)$ and using \eqref{disu} we get (see \cite[eq. (14)]{logbound}) 
$$au\le2c\log(u)\le 2c\log(1+u)\le \frac{2cu}{\sqrt{1+u}}.$$ 
Hence $a\sqrt{1+u} \le 2c$, which is to say that $u\le (2c/a)^2-1$, yielding the desired bound.
\end{proof}

\begin{lemma}\label{Lemma:magic}
Let $\{c_k\}_{k\ge 0}$ be a non-negative sequence and  $\xi>0$.
Then
\[
\sum_{j=0}^k  \frac{c_j}{\xi+ {\sum_{\ell=0}^j c_{\ell}}}
\le  \log\left(\frac{\xi + \sum_{j=0}^k c_j}{\xi} \right).
\]
\end{lemma}

\end{document}